\documentclass[12pt,a4paper]{article}
\usepackage{indentfirst}
\usepackage{amsmath,amssymb,amsfonts,amsthm,graphics}
\usepackage{makeidx}
\usepackage{color}

\newtheorem{theorem}{Theorem}

\newtheorem{lemma}{Lemma}
\newtheorem{claim}{Claim}

\begin{document}
\title{\Large\bf Asymptotic value of the minimal size of a graph with
rainbow connection number $2$\footnote{Supported by NSFC
No.11071130.}}
\author{\small Hengzhe Li, Xueliang Li, Yuefang Sun
\\
\small Center for Combinatorics and LPMC-TJKLC
\\
\small Nankai University, Tianjin 300071, China
\\
\small lhz2010@mail.nankai.edu.cn; lxl@nankai.edu.cn;
syf@cfc.nankai.edu.cn}
\date{}
\maketitle
\begin{abstract}
A path in an edge (vertex)-colored graph $G$, where adjacent edges
(vertices) may have the same color, is called a rainbow path if no
pair of edges (internal vertices) of the path are colored the same.
The rainbow (vertex) connection number $rc(G)$ ($rvc(G)$) of $G$ is
the minimum integer $i$ for which there exists an $i$-edge
(vertex)-coloring of $G$ such that every two distinct vertices of
$G$ are connected by a rainbow path. Denote by $\mathcal{G}_d(n)$
($\mathcal{G'}_d(n)$) the set of all graphs of order $n$ with
rainbow (vertex) connection number $d$, and define
$e_d(n)=\min\{e(G)\,|\, G\in \mathcal{G}_d(n)\}$
($e_d'(n)=\min\{e(G)\,|\, G\in \mathcal{G'}_d(n)\}$), where $e(G)$
denotes the number of edges in $G$. In this paper, we investigate
the bounds of $e_2(n)$ and get the exact asymptotic value. i.e.,
$\displaystyle \lim_{n\rightarrow \infty}\frac{e_2(n)}{n\log_2
n}=1$. Meanwhile, we obtain $e'_d(n)=n-1$ for $d\geq 2$, and the
equality holds if and only if $G$ is such a graph that deleting all
leaves of $G$ results in a tree of order $d$.

{\flushleft\bf Keywords}: edge(vertex)-coloring, rainbow path,
rainbow (vertex) connection number, diameter, minimal graph\\[2mm]
{\bf AMS subject classification 2010:} 05C15, 05C35, 05C40
\end{abstract}

\section{Introduction}

A communication network consists of nodes and links which connect
nodes. In order to prevent hackers, one can set a password in each
link (node). To facilitate the management, one can require that the
number of passwords is small enough such that any two nodes can
exchange information by a sequence of links (nodes) which have
different passwords. This problem can be modeled by a graph and
studied by means of rainbow (vertex) connection.

All graphs in this paper are undirected, finite and simple. We refer
to book \cite{bondy} for graph theoretical notation and terminology
not described here. A path in an edge (vertex)-colored graph $G$,
where adjacent edges (vertices) may have the same color, is called a
$rainbow\ path$ if no pair of edges (internal vertices) are colored
the same. An edge (vertex)-coloring of $G$ with $k$ colors is called
$k$-$rainbow$ if every two distinct vertices of $G$ are connected by
a rainbow path. The $rainbow\ (vertex)\ connection\ number\ rc(G)\
(rvc(G))$ of $G$ is the minimum integer $i$ for which there exists
an $i$-rainbow edge (vertex)-coloring of $G$ such that any two
distinct vertices of $G$ are connected by a rainbow path. It is easy
to see that $rc(G)\geq diam(G)$ and $rvc(G)\geq diam(G)-1$ for any
connected graph $G$, where $diam(G)$ is the diameter of $G$.

The rainbow connection number was introduced by Chartrand et al. in
\cite{char} which equivalents to the case that we set a password in
each link. They considered the rainbow connection numbers of several
graph classes (complete graphs, trees, cycles, wheels and complete
bipartite graphs) and showed the following result.

\begin{theorem}{\upshape\cite{char}}
$(i)$ $rc(G)=1$ if and only if $G$ is a complete graph.

$(ii)$ For integers $s$ and $t$ with $2\leq s\leq t,$
$$rc(K_{s,t})=\min\{\lceil\sqrt[s]{t}\rceil, 4\},$$
where $K_{s,t}$ is the complete bipartite graph with bipartition $X$
and $Y$, such that $|X|=s$ and $|Y|=t$.
\end{theorem}

Krivelevich and Yuster in \cite{kri}, and Schiermeyer in \cite{sch}
investigated the relation between the rainbow connection number and
the minimum degree of a graph. Chandran et al. \cite{kri} studied
the rainbow connection number of a graph by means of connected
dominating sets. Basavaraju et al. in \cite{bas} evaluated the
rainbow connection number of a graph by its radius and chordality
(size of a largest induced cycle). In \cite{chak}, Chakraborty et
al. investigated the hardness and algorithms for the rainbow
connection number, and got the following result.

\begin{theorem}
Given a graph $G$, deciding if $rc(G)=2$ is NP-Complete. In
particular, computing $rc(G)$ is NP-Hard.
\end{theorem}

It is well-known that almost all graphs have diameter $2$. In
\cite{lili}, Li et al. showed that $rc(G)\leq 5$ if $G$ is a
bridgeless graph of diameter $2$, and that $rc(G)\leq k+2$ if $G$ is
a connected graph of diameter $2$ with $k$ bridges, where $k\geq 1$.
For a detailed discussion regarding the origins of the problem,
practical applications and a survey of results, see \cite{lisun}.

Let $d$ and $n$ be natural numbers, $d<n$. Denote by
$\mathcal{G}_d(n)$ the set of all graphs of order $n$ with rainbow
connection number $d$. Define
$$e_d(n)=\min\{e(G)\,|\, G\in \mathcal{G}_d(n)\},$$
where $e(G)$ denotes the number of edges in $G$.

Because a network which satisfies our requirements and has as less
links as possible can cut costs, reduce the construction period and
simplify later maintenance, the study of this parameter is very
interesting and significant. In this paper, we investigate the lower
and upper bounds of $e_2(n)$ and get the exact asymptotic value for
the minimal size of a graph with rainbow connection number 2. The
following result is obtained:

\begin{theorem}
$$\lim_{n\rightarrow \infty}\frac{e_2(n)}{n\log_2 n}=1.$$
\end{theorem}

Krivelevich and Yuster in \cite{kri} introduced the concept of
rainbow vertex connection number which is equivalent to the case
that we set a password on each node. Let $d$ and $n$ be natural
numbers, $d<n$. Denote by $\mathcal{G'}_d(n)$ the set of all graphs
of order $n$ with rainbow vertex connection number $d$. Define
$$e_d'(n)=\min\{e(G)\,|\, G\in \mathcal{G}'_d(n)\},$$
where $e(G)$ denotes the number of edges in $G$. The following
result determines $e_d'(n)$.

\begin{theorem} Let $d$ be an integer larger than $1$. Then
$e_d'(n)=n-1$, and the equality holds if and only if $G$ is such a
graph that deleting all leaves of $G$ results in a tree of order
$d$.
\end{theorem}

In the next section, we will prove Theorems $2$ and $3$.

\section{The proofs of our main results}

Since result on the minimal graphs with respect to the rainbow
vertex connection number is easier to prove than that of the edge
case, we first show Theorem~$3$.

\noindent {\em Proof of Theorem~$3:$} Let $G$ be such a tree that
deleting all leaves of $G$ results in a tree $G'$ of order $d$. We
now give $G$ a $d$-rainbow vertex coloring as follows: color the
vertices of $G'$ by $d$ distinct color and color all leaves of $G$
by any used color. It is easy to check that this is a $d$-rainbow
vertex coloring. Thus $e_d'(n)\leq n-1$. On the other hand, any
connected graph has at least $n-1$ edges. Therefore $e'_d(n)=n-1$.

Now, we consider the second part of this theorem. The necessity
holds by the above argument. Conversely, let $G\in
\mathcal{G'}_d(n)$ with $e(G)=n-1$ and $c$ be a $d$-rainbow vertex
coloring. Suppose $G$ has $k$ leaves. Then we can give $G$ an
$(n-k)$-rainbow vertex coloring by the above argument. Thus
$rvc(G)=d\leq n-k$. On the other hand, we say $rvc(G)\geq n-k$. Let
$x$ and $y$ be any pair of vertices that are not leaves. Since $G$
is a tree, there exist two leaves, say $x',y'$, such that the unique
path between $x'$ and $y'$ in $G$ goes through $x$ and $y$. Thus
$c(x)\neq c(y)$. So $rvc(G)=d=n-k$, that is, $n-k=d$. Thus, by
deleting all leaves from $G$, we get a tree $G'$ with order $d$.
$\hfill\Box$

Now, we estimate the upper bound of $e_2(n)$ by constructing a
family of graphs.

\begin{lemma} For $n\geq 2$
$$e_2(n)\leq n\lceil\log_2 n\rceil-{(\lfloor\log_2 n\rfloor-1)}^2.$$
\end{lemma}

\begin{proof}
For each integer $n\geq 2$, there exists an integer $k$ such that
$k+2^{k-1}\leq n\leq k+2^k$. Consider the number of edges in the
complete bipartite graph $K_{k,n-k}$. We have

$$e(K_{k,n-k})=k(n-k)\leq n\lceil\log_2 n\rceil-{(\lfloor\log_2
n\rfloor-1)}^2.$$

Moreover, since $\lceil\sqrt[k]{n-k}\rceil\leq
\lceil\sqrt[k]{2^k}\rceil=2$, $rc(K_{k,n-k})=2$ follows from
Theorem~$1$. Thus $e_2(n)\leq n\lceil\log_2 n\rceil-{(\lfloor\log_2
n\rfloor-1)}^2$.
\end{proof}

Consider a graph $G\in \mathcal{G}_2(n)$ with order $n$ and maximum
degree $\Delta$. Pick a vertex $u\in V(G)$. Since $d(u)\leq \Delta$,
there exist at most $\Delta$ vertices adjacent to $u$, and at most
$\Delta(\Delta-1)$ vertices at distance $2$ from $u$. Since
$diam(G)=2$, we derive $n\leq 1+\Delta+\Delta(\Delta-1)$. Thus
$\Delta \geq \sqrt{n-1}$. Since $\Delta$ is an integer, we get

\begin{equation}
 \Delta \geq \lceil\sqrt{n-1}\rceil.
\end{equation}

Next, we consider to get a lower bound for $e_2(n)$.

\begin{lemma}
$(i)$ $$e_2(n)\geq \min\{\frac{n}{2}\log_2 n,\,n\log_2 n-4n\}.$$

$(ii)$ If $n\geq 2^{17}$, then
$$e_2(n)\geq n\log_2 n-2n.$$
\end{lemma}

\begin{proof}
Let $G$ be a graph with diameter $2$ and $c$ be a $2$-rainbow
edge-coloring of $G$ with colors blue and red. Set $k=\lceil{(\log_2
\sqrt{n})}^2\rceil$ and denote by $S$ the set of vertices with
degrees less than $k$. Assume $S=\{u_1,u_2,\ldots,u_s\},
T=V(G)\backslash S=\{u_{s+1},u_{s+2},\ldots,u_{s+t}\}$, where
$s+t=n$. By $(1)$ and $k=\lceil{(\log_2 \sqrt{n})}^2\rceil\leq
\lceil\sqrt{n-1}\rceil\leq \Delta$, we know that $T$ is nonempty. If
$t=|T|\geq \frac{2n}{\log_2\sqrt{n}}$, then

$$e(G)\geq \frac{1}{2}\sum_{v\in T}d_{G}(v)\geq\frac{1}{2}\frac{2n}{\log_2\sqrt{n}}
\lceil{(\log_2\sqrt{n})}^2\rceil\geq \frac{n}{2}\log_2 n,$$

\noindent we are done.

Suppose $t<\frac{2n}{\log_2\sqrt{n}}$, that is,
$s>n-\frac{2n}{\log_2\sqrt{n}}$. Clearly, it is sufficient to show
that $e(S,T)\geq n\log_2 n-4n.$

For every $u_i,\, 1\leq i\leq s$, we define a vector as follows:

$$\alpha(u_i)=(b_{i,1},b_{i,2},\ldots,b_{i,t}),$$ where
  \begin{equation*} b_{i,j}=
  \begin{cases} 1 & if\ c(u_iu_{s+j})\ is\ red;\\
  -1 & if\ c(u_iu_{s+j})\ is\ blue;\\
  0  & if\ u_i\ and\ u_{s+j}\ is\ nonadjacent.
  \end{cases}
  \end{equation*}

Suppose $|N(u_i)\cap T|=a_i,\, 1\leq i\leq s$. Then
$e(S,T)=\sum_{i=1}^{s}a_i$, where $e(S,T)$ denotes the number of
edges between $S$ and $T$. We now estimate the value of $e(S,T)$.
For each $\alpha(u_i)$, we define a set $B_i$ as follows:
$B_i=\{$vectors obtained from $\alpha(u_i)$ by replace ``$0$'' of
$\alpha(u_i)$ by ``$1$'' or ``$-1$''$\}$. Because $|N(u_i)\cap
T|=a_i$, we have $|B_i|=2^{t-a_i}$ for each $i$, where $1\leq i\leq
s$. Set $B=\bigcup_{i=1}^s B_i$. Then $B$ is a multiset of
$t$-dimensional vectors with elements $1$ and $-1$. For each
$\alpha\in B$, $n_{\alpha}$ denotes the number of $\alpha$ in $B$.
We have the following claim.

\begin{claim}
For each $\alpha\in B$, $n_{\alpha}\leq k^2+1$.
\end{claim}

\noindent{\em Proof of Claim~$1:$} If Claim~$1$ is not true, that
is, there exists a vector $\alpha$, without loss of generality,
assume $\alpha=(b_1,b_2,\ldots,b_t)$, such that $n_{\alpha}\geq
k^2+2$. Clearly, it is not possible that there exists some $B_i$
such that $B_i$ contains two $\alpha$. Thus, there exist $k^2+2$
integers, without loss of generality, say $1,2,\ldots,k^2+2$, such
that $\alpha\in B_r,\, 1\leq r\leq k^2+2$. we next show that for
each $i,\, 2\leq i\leq k^2+2$, the distance between $u_1$ and $u_i$
in $G[S]$ is at most $2$. In fact, $c(u_1u_{s+j})=b_j=c(u_iu_{s+j})$
follows from the definition of $B_1$ and $B_i$. Thus there exists no
rainbow path between $u_1$ and $u_i$ through a vertex contained in
$T$. So there must exist a rainbow path between $u_1$ and $u_i$ with
length at most $2$ in $G[S]$. On the other hand, since
$\Delta(G[S])\leq k$, the number of vertices at distance $2$ from
$u_1$ is at most $k^2+1$, which is a contradiction. So, this claim
is true.

By Claim~$1$, we know
$$\sum_{i=1}^s|B_i|\leq (k^2+1)2^t,$$
Since $|B_i|=2^{t-a_i}$ for each $i$, $1\leq i\leq s$,
$$\sum_{i=1}^s 2^{-a_i}\leq (k^2+1).$$
By the inequality between the geometrical and arithmetical means, we
have

$${\large\sqrt[s]{2^{-e(S,T)}}}={\large\sqrt[s]{2^{-\sum_{i}^s a_i}}}\leq
\frac{1}{s}\sum_{i=1}^s 2^{-a_i}\leq \frac{k^2+1}{s},$$

\noindent using the log function on both sides,

$$e(S,T)\geq s\log_2 s-s\log_2 (k^2+1).$$

Since $e(S,T)$ is monotonically decreasing in $s$ and
$s>n-\frac{2n}{\log_2\sqrt{n}}$, we have
\begin{eqnarray*}
e(S,T)&\geq& (n-\frac{2n}{\log_2\sqrt{n}})\log_2
(n-\frac{2n}{\log_2\sqrt{n}})\\[6pt]
&&{}-(n-\frac{2n}{\log_2\sqrt{n}})\log_2 ({\lceil{(\log_2
\sqrt{n})}^2\rceil}^2+1)\\[6pt]
&=& n\log_2 n-4n.
\end{eqnarray*}

For $(ii)$, take $k=\lceil{(\log_2 n)}^2\rceil$. Since $n\geq
2^{17}$, we have $\lceil{(\log_2 n)}^2\rceil\leq
\lceil\sqrt{n-1}\rceil$. Thus $T$ is nonempty by Ineq. $(1)$. All
the remaining arguments are similar to $(i)$.

This completes the proof.
\end{proof}

Combining Lemmas~$1$ and $2$, we know that Theorem $2$ holds.

\end{document}